\documentclass[11pt,a4paper]{article}
\usepackage[T1]{fontenc}

\usepackage{amsmath,amssymb,amsthm,graphicx}
\usepackage[subrefformat=parens,labelformat=parens]{subfig}
\usepackage[all]{xy}   
\usepackage{verbatim}
\usepackage[dvipsnames]{xcolor}
\usepackage{tikz}
\usetikzlibrary{positioning,shapes,shadows}
\usetikzlibrary{arrows,automata,lindenmayersystems}

\usepackage{ulem}

\usepackage[shortlabels]{enumitem}

\usetikzlibrary{decorations.pathreplacing,shapes.multipart}

\makeatletter
\DeclareRobustCommand{\rvdots}{%
	\vbox{
		\baselineskip4\p@\lineskiplimit\z@
		\kern-\p@
		\hbox{.}\hbox{.}\hbox{.}
}}
\makeatother

\setlist{nosep}

\usepackage{microtype}

\definecolor{darkblue}{rgb}{0.0, 0.0, 0.55}
\usepackage{hyperref}

\usepackage[noabbrev,nameinlink,capitalise]{cleveref}
        \crefname{subsection}{Subsection}{Subsections}
        \crefname{subsection}{Subsection}{Subsections}
\crefname{theorem}{Theorem}{ }
\crefname{corollary}{Corollary}{ }
\crefname{lemma}{Lemma}{ }
\crefname{remark}{Remark}{ }
\crefname{section}{Section}{ }
\crefname{subsection}{Subsection}{ }
\crefname{figure}{Figure}{ }

\usepackage{thmtools, thm-restate}

\crefalias{theoremx}{theorem}

\DeclareEmphSequence{\bfseries}

\newtheorem{theorem}{Theorem}[section]
\newtheorem{proposition}[theorem]{Proposition}
\newtheorem{lemma}[theorem]{Lemma}

\newtheorem{conjecture}[theorem]{Conjecture}

\theoremstyle{definition}
\newtheorem{definition}[theorem]{Definition}
\newtheorem*{definition*}{Definition}

\newtheorem{remark}[theorem]{Remark}

\newcommand{\diam}{\mathrm{diam}}
\newcommand{\dist}{\mathrm{d}}

\usepackage[style=alphabetic,maxbibnames=99]{biblatex}
        \DeclareFieldFormat{title}{\mkbibquote{#1}}
        \DeclareFieldFormat{journal}{\mkbibquote{#1}}
        \DeclareFieldFormat{booktitle}{\mkbibquote{#1}}

    \newtheoremstyle{TheoremNum}
        {\topsep}{\topsep}              
        {\itshape}                      
        {}                              
        {\bfseries}                     
        {.}                             
        { }                             
        {\thmname{#1}\thmnote{ \bfseries #3}}
    \theoremstyle{TheoremNum}

\tikzset{initial text={}}
\makeatletter
\tikzstyle{initial by arrow}=   [after node path=
{
  {
    [to path=
    {
      [->,double=none,every initial by arrow]
      ([shift=(\tikz@initial@angle:\tikz@initial@distance)]\tikztostart.\tikz@initial@angle)
          node [shape=rectangle,anchor=\tikz@initial@anchor] {\tikz@initial@text}
        -- (\tikztostart)}]
    edge ()
  }
}]
\makeatother

\begin{document}

\title{Horofunctions of infinite Sierpinski polygon graphs}

\author{
Daniele D'Angeli %
\footnote{Universit\`a Niccolo Cusano, Rome, Italy \texttt{daniele.dangeli@unicusano.it}}
\and Francesco Matucci%
\footnote{Universit\`{a} 
di Milano--Bicocca, Milan, Italy.
\texttt{francesco.matucci@unimib.it}}
\and Davide Perego%
\footnote{Université de Genève, Geneva, Switzerland.
\texttt{dperego9@gmail.com}}
\and Emanuele Rodaro%
\footnote{
Politecnico di Milano, Milan, Italy.
\texttt{emanuele.rodaro@polimi.it}}}

\maketitle

\begin{abstract}
Generalizing works of D'Angeli and Donno, we describe, starting from an infinite sequence over $r$ letters with $r\neq4i$ and $i\in \mathbb{N}$, a sequence of pointed finite graphs. We study the pointed Gromov–Hausdorff limit graphs giving a description of isomorphim classes in terms of dihedral groups and providing insights on the horofunction boundaries in terms of Busemann and non-Busemann points.
\end{abstract}

\section*{Introduction}

The horofunction boundary is a powerful compactification of a metric space based on analytic tools (\cite{Gromov81}). What makes it interesting and particularly suitable to study is, besides its generality, a geometric description of horofunctions as \textit{weakly-geodesic rays} (\cite{Rieffel02}). A particular effort has been made trying to understand the subset of Busemann points (\cite{Gromov81}), that is, horofunctions which can be described via geodesic rays. Among many different aspects, there have been investigations regarding its connections to dynamical systems (\cite{CP,Pawlik,Perego}), group theory (\cite{BT24,TY16}) and graph theory (\cite{WW06}). \\

The aim of this paper is to generalize and complete the work of the first author (\cite{D17}) on Sierpinski triangle graphs. In particular, we consider Sierpinski polygon graphs where the perimeter of the polygon is not a multiple of $4$. They are Gromov-Hausdorff limits of recursively generated finite graphs inspired by the Sierpinski polygon fractals. The techniques reflect elements from the first author and Donno (\cite{DD14},\cite{DD16}), where Sierpinski gasket graphs were taken into account. Notice 
that in the case of the Sierpinski gasket ($r=4$ in the notation that we will use later),
the corresponding fractal is not finitely ramified, and that is why these strategies cannot be directly applied. The same holds for any $r=4i$ and $i \in \mathbb{N}$.\\

In \cref{sec:1} we give a proper definition of the graphs $\Gamma_\xi^{(r)}$, where $\xi$ is an infinite sequence over a finite alphabet of $r$ letters. The sequence helps describing the recursion process which defines the graph. We then address the isomorphism problem in \cref{sec:2}.\\

\noindent \medskip
\textbf{\cref{thm:iso}.}
\textit{Let $D_r$ be dihedral group of the $r$-polygon. Two graphs $\Gamma_\xi^{(r)}$ and $\Gamma_\eta^{(r)}$ are isomorphic if and only if there exist $\sigma\in D_r$ such that $\eta$ and $\sigma(\xi)$ are cofinal.}

\medskip
In order to prove it, we exploit the aforementioned Gromov-Hausdorff convergence, together with some considerations which highlight a connection with Schreier and tile graphs of self-similar groups (\cite{BDN}).\\

In \cref{sec:3} we analyze the structure of cut points in the graphs and give a countable collection of non-Busemann points and a finite collection of Busemann points when $\Gamma_\xi^{(r)}$ is eventually constant, resulting in the following:\\

\medskip
\noindent \textbf{\cref{mainthm}.}
\textit{The horofunction boundary of $\Gamma_\xi^{(r)}$ with $\xi=wj^{\infty}$ and $w$ a finite prefix, contains two Busemann points and countably many non-Busemann points.}

\section{Sierpinski polygon graphs}\label{sec:1}

Let $r$ be a positive integer that is not a multiple of $4$. Consider the cyclic graph $P_r$ with $r$ vertices and denote them $V_r=\{0,1,\ldots, r-1\}$. Define $f(r):=\min \{ i \mid 4i >r\}$ and $\widetilde{f}(r):=2f(r)$.

The \textbf{Sierpinski graphs} represent a sequence of graphs inspired by the Sierpinski fractals. The $r-$th \textbf{Sierpinski sequence} $\{\Gamma_k^{(r)}: k\geq 1\}$ is recursively constructed as follows:

\begin{enumerate}
    \item $\Gamma_1^{(r)}$ coincides with $P_r$ whose vertex set is $\{0,1,\ldots, r-1\}$;\\
    \item $\Gamma_k^{(r)}$ is obtained by taking $r$ copies of $\Gamma_{k-1}^{(r)}$ numbered from 0 to $r-1$. A vertex $v$ in the $i-$th copy is labeled $vi$. 
    Identify, by gluing the copies, exactly two vertices for each $\Gamma_{k-1}^{(r)}$. One will be identified with a vertex of the $(i+1)$-copy and the other with a vertex in the $(i-1)$-copy.
    In particular,

    $$(i+f(r))^{k-1}i \mod r \ \ \text{  identifies with  } \ \ ((i+1)+\widetilde{f}(r))^{k-1}(i+1) \mod r.$$
\end{enumerate}

\noindent For example if $r=6$, then $f(6)=2$ and $\widetilde{f}(6)=4$ and hence $004$ and $335$ identify in $\Gamma_3^{(6)}$, as well as $115$ and $440$ (see green vertices in \cref{fig:hexagon-gluing}). \

Before going on, let us note that such type of recursive constructions are already present in literature. For example, in theory of automata groups, in particular when describing Schreier graphs (see e.g. \cite{BDN}), and in rational dynamical systems where such graphs are associated to iterations of the system (see \cite{HP}). In the first case, these analogies persist (see further details below), whereas in the second case they cease, as the graphs are intended to approximate the standard Sierpinski fractals.\

The graph $\Gamma_k^{(r)}$ can be seen as a pointed graph $(\Gamma_k^{(r)},w)$, once we choose a sequence $w$ of length $k$ over $V_r$ (see e.g. \cref{fig:hexagon-gluing}).  The choice of the notation for the sequences and the fractal construction growing instead of remaining of the same size is justified by the next definitions.

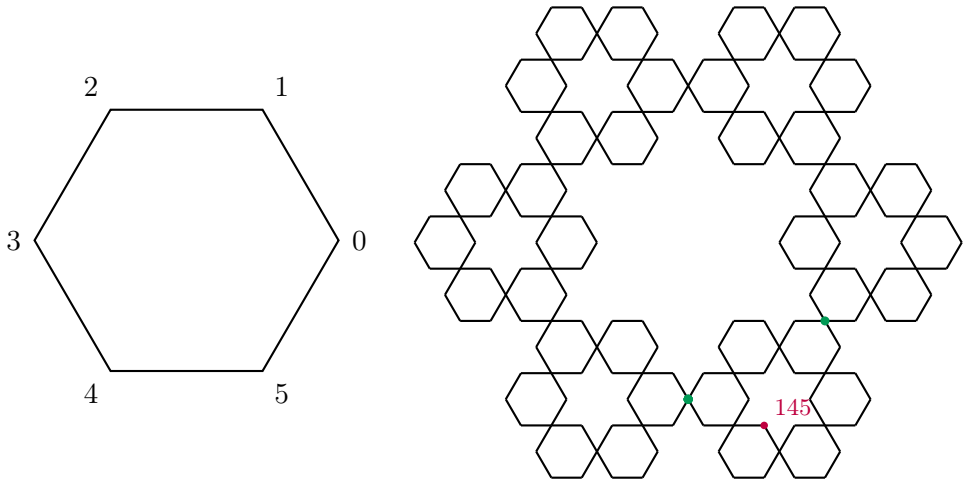
\begin{figure}
    \centering
\begin{tikzpicture}
\begin{scope}[scale=0.2]
\begin{scope} 
  \foreach \i in {0,60,...,300} {
    \draw[thick] (\i:2) -- (\i+60:2);
  }
\end{scope}
\begin{scope}[xshift=4cm] 
  \foreach \i in {0,60,...,300} {
    \draw[thick] (\i:2) -- (\i+60:2);
  }
\end{scope}
\begin{scope}[xshift=-2cm,yshift=3.464cm] 
  \foreach \i in {0,60,...,300} {
    \draw[thick] (\i:2) -- (\i+60:2);
  }
\end{scope}
\begin{scope}[yshift=6.928cm] 
  \foreach \i in {0,60,...,300} {
    \draw[thick] (\i:2) -- (\i+60:2);
  }
\end{scope}
\begin{scope}[xshift=4cm,yshift=6.928cm] 
  \foreach \i in {0,60,...,300} {
    \draw[thick] (\i:2) -- (\i+60:2);
  }
\end{scope}
\begin{scope}[xshift=6cm,yshift=3.464cm] 
  \foreach \i in {0,60,...,300} {
    \draw[thick] (\i:2) -- (\i+60:2);
  }
\end{scope} 
\end{scope}

\begin{scope}[scale=0.2,xshift=12cm] 
\begin{scope} 
  \foreach \i in {0,60,...,300} {
    \draw[thick] (\i:2) -- (\i+60:2);
  }

  \draw[purple] (60:2) node[circle,inner sep=1pt,fill]{} node[above right]{\footnotesize 145};
\end{scope}
\begin{scope}[xshift=4cm] 
  \foreach \i in {0,60,...,300} {
    \draw[thick] (\i:2) -- (\i+60:2);
  }
\end{scope}
\begin{scope}[xshift=-2cm,yshift=3.464cm] 
  \foreach \i in {0,60,...,300} {
    \draw[thick] (\i:2) -- (\i+60:2);
    \draw[ForestGreen] (180:2) node[circle,inner sep=1.2pt,fill]{} node[above right]{};
  }
\end{scope}
\begin{scope}[yshift=6.928cm] 
  \foreach \i in {0,60,...,300} {
    \draw[thick] (\i:2) -- (\i+60:2);
  }
\end{scope}
\begin{scope}[xshift=4cm,yshift=6.928cm] 
  \foreach \i in {0,60,...,300} {
    \draw[thick] (\i:2) -- (\i+60:2);
  }
\end{scope}
\begin{scope}[xshift=6cm,yshift=3.464cm] 
  \foreach \i in {0,60,...,300} {
    \draw[thick] (\i:2) -- (\i+60:2);
  }
\end{scope} 
\end{scope}

\begin{scope}[scale=0.2,xshift=-6cm,yshift=10.392cm] 
\begin{scope} 
  \foreach \i in {0,60,...,300} {
    \draw[thick] (\i:2) -- (\i+60:2);
  }
\end{scope}
\begin{scope}[xshift=4cm] 
  \foreach \i in {0,60,...,300} {
    \draw[thick] (\i:2) -- (\i+60:2);
  }
\end{scope}
\begin{scope}[xshift=-2cm,yshift=3.464cm] 
  \foreach \i in {0,60,...,300} {
    \draw[thick] (\i:2) -- (\i+60:2);
  }
\end{scope}
\begin{scope}[yshift=6.928cm] 
  \foreach \i in {0,60,...,300} {
    \draw[thick] (\i:2) -- (\i+60:2);
  }
\end{scope}
\begin{scope}[xshift=4cm,yshift=6.928cm] 
  \foreach \i in {0,60,...,300} {
    \draw[thick] (\i:2) -- (\i+60:2);
  }
\end{scope}
\begin{scope}[xshift=6cm,yshift=3.464cm] 
  \foreach \i in {0,60,...,300} {
    \draw[thick] (\i:2) -- (\i+60:2);
  }
\end{scope} 
\end{scope}

\begin{scope}[scale=0.2,yshift=20.784cm] 
\begin{scope} 
  \foreach \i in {0,60,...,300} {
    \draw[thick] (\i:2) -- (\i+60:2);
  }
\end{scope}
\begin{scope}[xshift=4cm] 
  \foreach \i in {0,60,...,300} {
    \draw[thick] (\i:2) -- (\i+60:2);
  }
\end{scope}
\begin{scope}[xshift=-2cm,yshift=3.464cm] 
  \foreach \i in {0,60,...,300} {
    \draw[thick] (\i:2) -- (\i+60:2);
  }
\end{scope}
\begin{scope}[yshift=6.928cm] 
  \foreach \i in {0,60,...,300} {
    \draw[thick] (\i:2) -- (\i+60:2);
  }
\end{scope}
\begin{scope}[xshift=4cm,yshift=6.928cm] 
  \foreach \i in {0,60,...,300} {
    \draw[thick] (\i:2) -- (\i+60:2);
  }
\end{scope}
\begin{scope}[xshift=6cm,yshift=3.464cm] 
  \foreach \i in {0,60,...,300} {
    \draw[thick] (\i:2) -- (\i+60:2);
  }
\end{scope} 
\end{scope}

\begin{scope}[scale=0.2,xshift=12cm,yshift=20.784cm] 
\begin{scope} 
  \foreach \i in {0,60,...,300} {
    \draw[thick] (\i:2) -- (\i+60:2);
  }
\end{scope}
\begin{scope}[xshift=4cm] 
  \foreach \i in {0,60,...,300} {
    \draw[thick] (\i:2) -- (\i+60:2);
  }
\end{scope}
\begin{scope}[xshift=-2cm,yshift=3.464cm] 
  \foreach \i in {0,60,...,300} {
    \draw[thick] (\i:2) -- (\i+60:2);
  }
\end{scope}
\begin{scope}[yshift=6.928cm] 
  \foreach \i in {0,60,...,300} {
    \draw[thick] (\i:2) -- (\i+60:2);
  }
\end{scope}
\begin{scope}[xshift=4cm,yshift=6.928cm] 
  \foreach \i in {0,60,...,300} {
    \draw[thick] (\i:2) -- (\i+60:2);
  }
\end{scope}
\begin{scope}[xshift=6cm,yshift=3.464cm] 
  \foreach \i in {0,60,...,300} {
    \draw[thick] (\i:2) -- (\i+60:2);
  }
\end{scope} 
\end{scope}

\begin{scope}[scale=0.2,xshift=18cm,yshift=10.392cm] 
\begin{scope} 
  \foreach \i in {0,60,...,300} {
    \draw[thick] (\i:2) -- (\i+60:2);
  }
  \draw[ForestGreen] (240:2) node[circle,inner sep=1.2pt,fill]{} node[above right]{};
\end{scope}
\begin{scope}[xshift=4cm] 
  \foreach \i in {0,60,...,300} {
    \draw[thick] (\i:2) -- (\i+60:2);
  }
\end{scope}
\begin{scope}[xshift=-2cm,yshift=3.464cm] 
  \foreach \i in {0,60,...,300} {
    \draw[thick] (\i:2) -- (\i+60:2);
  }
\end{scope}
\begin{scope}[yshift=6.928cm] 
  \foreach \i in {0,60,...,300} {
    \draw[thick] (\i:2) -- (\i+60:2);
  }
\end{scope}
\begin{scope}[xshift=4cm,yshift=6.928cm] 
  \foreach \i in {0,60,...,300} {
    \draw[thick] (\i:2) -- (\i+60:2);
  }
\end{scope}
\begin{scope}[xshift=6cm,yshift=3.464cm] 
  \foreach \i in {0,60,...,300} {
    \draw[thick] (\i:2) -- (\i+60:2);
  }
\end{scope} 
\end{scope}

\begin{scope}[xshift=-5cm,yshift=2.8cm]
  \foreach \i in {0,...,5} {
    \coordinate (P\i) at ({60*\i}:2);
  }

  \draw[thick] (P0) -- (P1) -- (P2) -- (P3) -- (P4) -- (P5) -- cycle;

  \foreach \i in {0,...,5} {
    \node[inner sep=1pt, label={60*\i:\i}] at (P\i) {};
  }
\end{scope}

\end{tikzpicture}

\caption{The cyclic graph $P_6$ and the corresponding pointed graph $(\Gamma_3^{(6)},145)$ with two gluing vertices highlighted.}
    \label{fig:hexagon-gluing}
\end{figure}

\begin{definition}
A sequence of pointed graphs $\{(X_{m},p_{m})\}_{m=1}^{\infty}$ \textbf{strongly converges in the Gromov-Hausdorff sense} to a pointed graph $(X,p)$ if for every $r >0$ there exists $M(r)>0$ such that for every $m>M(r)$ there is an isometry between the two balls $B_{X_{m}}(p_{m},r)$ and $B_{X}(p,r)$ sending $p_{m}$ to $p$. 
\end{definition}
Note that this convergence implies the classic Gromov-Hausdorff convergence for pointed metric spaces (see e.g. \cite[Definition 8.1.1]{BBI}).

In our case, we can consider an infinite sequence $\xi \in V_r^{\infty}$ and denote by $\xi[k]$ the prefix of length $k$ of $\xi$. In this way, we define $(\Gamma_\xi^{(r)}, \xi)$ to be the limit of the  strong Gromov-Hausdorff convergence
$$
(\Gamma_k^{(r)}, \xi[k])\to (\Gamma_\xi^{(r)}, \xi).
$$
See \cref{fig:limit-graph} for an example with $r=6$ and $\xi=1\overline{54}$. The intuition of the Gromov-Hausdorff convergence is that, once we have fixed a sequence $\xi$, our infinite graph is growing around it and the $k$-digit of the sequence is telling us the position of the $k$-step finite graph inside $\Gamma_{k+1}$. An analog can be witnessed in the theory of infinite Schreier graphs of bounded automata groups, where the vertices are infinite sequences and the neighborhood of a vertex is determined by a finite Schreier graph of a prefix (see e.g. \cite{BDN}). Following this point of view, another resemblance between the two cases can be spotted in \cref{lem:in}.
Note that, depending on the infinite sequence we choose, we might end up with different limit graphs $(\Gamma_\xi^{(r)}, \xi)$. 
To avoid cumbersome notation, we will simply refer to $\Gamma_{k}$  and $\Gamma_\xi$, dropping the $r$. Also, we will write $B_{k}(\xi[k],\ell)=B_{\Gamma_k}(\xi[k],\ell)$ and $B(\xi,\ell)=B_{\Gamma_\xi}(\xi,\ell)$ \\

We will refer to vertices of $\Gamma_\xi$ belonging to two different copies of $\Gamma_k$ as \textbf{gluing vertices}.
\\

\begin{figure}
    \centering
\begin{tikzpicture}
\begin{scope}[scale=0.2]
\begin{scope} 
  \foreach \i in {0,60,...,300} {
    \draw[thick] (\i:2) -- (\i+60:2);
  }
\end{scope}
\begin{scope}[xshift=4cm] 
  \foreach \i in {0,60,...,300} {
    \draw[thick] (\i:2) -- (\i+60:2);
  }
\end{scope}
\begin{scope}[xshift=-2cm,yshift=3.464cm] 
  \foreach \i in {0,60,...,300} {
    \draw[thick] (\i:2) -- (\i+60:2);
  }
\end{scope}
\begin{scope}[yshift=6.928cm] 
  \foreach \i in {0,60,...,300} {
    \draw[thick] (\i:2) -- (\i+60:2);
  }
\end{scope}
\begin{scope}[xshift=4cm,yshift=6.928cm] 
  \foreach \i in {0,60,...,300} {
    \draw[thick] (\i:2) -- (\i+60:2);
  }
\end{scope}
\begin{scope}[xshift=6cm,yshift=3.464cm] 
  \foreach \i in {0,60,...,300} {
    \draw[thick] (\i:2) -- (\i+60:2);
  }
  \draw[purple] (240:2) node[circle,inner sep=1pt,fill]{} node[above left]{\footnotesize $1\overline{54}$ \ };
\end{scope} 
\end{scope}

\begin{scope}[scale=0.2,xshift=12cm] 
\begin{scope} 
  \foreach \i in {0,60,...,300} {
    \draw[thick] (\i:2) -- (\i+60:2);
  }
\end{scope}
\begin{scope}[xshift=4cm] 
  \foreach \i in {0,60,...,300} {
    \draw[thick] (\i:2) -- (\i+60:2);
  }
\end{scope}
\begin{scope}[xshift=-2cm,yshift=3.464cm] 
  \foreach \i in {0,60,...,300} {
    \draw[thick] (\i:2) -- (\i+60:2);
  }
\end{scope}
\begin{scope}[yshift=6.928cm] 
  \foreach \i in {0,60,...,300} {
    \draw[thick] (\i:2) -- (\i+60:2);
  }
\end{scope}
\begin{scope}[xshift=4cm,yshift=6.928cm] 
  \foreach \i in {0,60,...,300} {
    \draw[thick] (\i:2) -- (\i+60:2);
  }
\end{scope}
\begin{scope}[xshift=6cm,yshift=3.464cm] 
  \foreach \i in {0,60,...,300} {
    \draw[thick] (\i:2) -- (\i+60:2);
  }
\end{scope} 
\end{scope}

\begin{scope}[scale=0.2,xshift=-6cm,yshift=10.392cm] 
\begin{scope} 
  \foreach \i in {0,60,...,300} {
    \draw[thick] (\i:2) -- (\i+60:2);
  }
\end{scope}
\begin{scope}[xshift=4cm] 
  \foreach \i in {0,60,...,300} {
    \draw[thick] (\i:2) -- (\i+60:2);
  }
\end{scope}
\begin{scope}[xshift=-2cm,yshift=3.464cm] 
  \foreach \i in {0,60,...,300} {
    \draw[thick] (\i:2) -- (\i+60:2);
  }
  \draw[dashed,thick] (180:2)++(240:2.1) -- (180:2);
  \draw[dashed,thick] (180:2)++(120:2.1) -- (180:2);
  \draw[] (180:2)++(180:2.6) node[]{\footnotesize $\ldots5\ldots$ \ };
\end{scope}
\begin{scope}[yshift=6.928cm] 
  \foreach \i in {0,60,...,300} {
    \draw[thick] (\i:2) -- (\i+60:2);
  }
\end{scope}
\begin{scope}[xshift=4cm,yshift=6.928cm] 
  \foreach \i in {0,60,...,300} {
    \draw[thick] (\i:2) -- (\i+60:2);
  }
\end{scope}
\begin{scope}[xshift=6cm,yshift=3.464cm] 
  \foreach \i in {0,60,...,300} {
    \draw[thick] (\i:2) -- (\i+60:2);
  }
\end{scope} 
\end{scope}

\begin{scope}[scale=0.2,yshift=20.784cm] 
\begin{scope} 
  \foreach \i in {0,60,...,300} {
    \draw[thick] (\i:2) -- (\i+60:2);
  }
\end{scope}
\begin{scope}[xshift=4cm] 
  \foreach \i in {0,60,...,300} {
    \draw[thick] (\i:2) -- (\i+60:2);
  }
\end{scope}
\begin{scope}[xshift=-2cm,yshift=3.464cm] 
  \foreach \i in {0,60,...,300} {
    \draw[thick] (\i:2) -- (\i+60:2);
  }
\end{scope}
\begin{scope}[yshift=6.928cm] 
  \foreach \i in {0,60,...,300} {
    \draw[thick] (\i:2) -- (\i+60:2);
  }
  \draw[dashed,thick] (120:2)++(60:2.1) -- (120:2); 
  \draw[dashed,thick] (120:2)++(180:2.1) -- (120:2);
  \draw[] (120:2)++(120:3) node[rotate=-60]{\footnotesize $\ldots4\ldots$ \ };
\end{scope}
\begin{scope}[xshift=4cm,yshift=6.928cm] 
  \foreach \i in {0,60,...,300} {
    \draw[thick] (\i:2) -- (\i+60:2);
  }
\end{scope}
\begin{scope}[xshift=6cm,yshift=3.464cm] 
  \foreach \i in {0,60,...,300} {
    \draw[thick] (\i:2) -- (\i+60:2);
  }
\end{scope} 
\end{scope}

\begin{scope}[scale=0.2,xshift=12cm,yshift=20.784cm] 
\begin{scope} 
  \foreach \i in {0,60,...,300} {
    \draw[thick] (\i:2) -- (\i+60:2);
  }
\end{scope}
\begin{scope}[xshift=4cm] 
  \foreach \i in {0,60,...,300} {
    \draw[thick] (\i:2) -- (\i+60:2);
  }
\end{scope}
\begin{scope}[xshift=-2cm,yshift=3.464cm] 
  \foreach \i in {0,60,...,300} {
    \draw[thick] (\i:2) -- (\i+60:2);
  }
\end{scope}
\begin{scope}[yshift=6.928cm] 
  \foreach \i in {0,60,...,300} {
    \draw[thick] (\i:2) -- (\i+60:2);
  }
\end{scope}
\begin{scope}[xshift=4cm,yshift=6.928cm] 
  \foreach \i in {0,60,...,300} {
    \draw[thick] (\i:2) -- (\i+60:2);
  }
  \draw[dashed,thick] (60:2)++(0:2.1) -- (60:2); 
  \draw[dashed,thick] (60:2)++(120:2.1) -- (60:2);
  \draw[] (60:2)++(60:3) node[rotate=60]{\footnotesize $\ldots5\ldots$ \ };
\end{scope}
\begin{scope}[xshift=6cm,yshift=3.464cm] 
  \foreach \i in {0,60,...,300} {
    \draw[thick] (\i:2) -- (\i+60:2);
  }
\end{scope} 
\end{scope}

\begin{scope}[scale=0.2,xshift=18cm,yshift=10.392cm] 
\begin{scope} 
  \foreach \i in {0,60,...,300} {
    \draw[thick] (\i:2) -- (\i+60:2);
  }
  \draw[ForestGreen] (240:2) node[circle,inner sep=1.2pt,fill]{} node[above right]{};
\end{scope}
\begin{scope}[xshift=4cm] 
  \foreach \i in {0,60,...,300} {
    \draw[thick] (\i:2) -- (\i+60:2);
  }
\end{scope}
\begin{scope}[xshift=-2cm,yshift=3.464cm] 
  \foreach \i in {0,60,...,300} {
    \draw[thick] (\i:2) -- (\i+60:2);
  }
\end{scope}
\begin{scope}[yshift=6.928cm] 
  \foreach \i in {0,60,...,300} {
    \draw[thick] (\i:2) -- (\i+60:2);
  }
\end{scope}
\begin{scope}[xshift=4cm,yshift=6.928cm] 
  \foreach \i in {0,60,...,300} {
    \draw[thick] (\i:2) -- (\i+60:2);
  }
\end{scope}
\begin{scope}[xshift=6cm,yshift=3.464cm] 
  \foreach \i in {0,60,...,300} {
    \draw[thick] (\i:2) -- (\i+60:2);
  }
  \draw[dashed,thick] (0:2)++(300:2.1) -- (0:2);
  \draw[dashed,thick] (0:2)++(60:2.1) -- (0:2); 
  \draw[] (0:2)++(0:3) node[]{\footnotesize $\ldots4\ldots$ \ };
\end{scope} 
\end{scope}
\end{tikzpicture}
    \caption{A portion of the graph $\Gamma_{1\overline{54}}^{(6)}$.}
    \label{fig:limit-graph}
\end{figure}
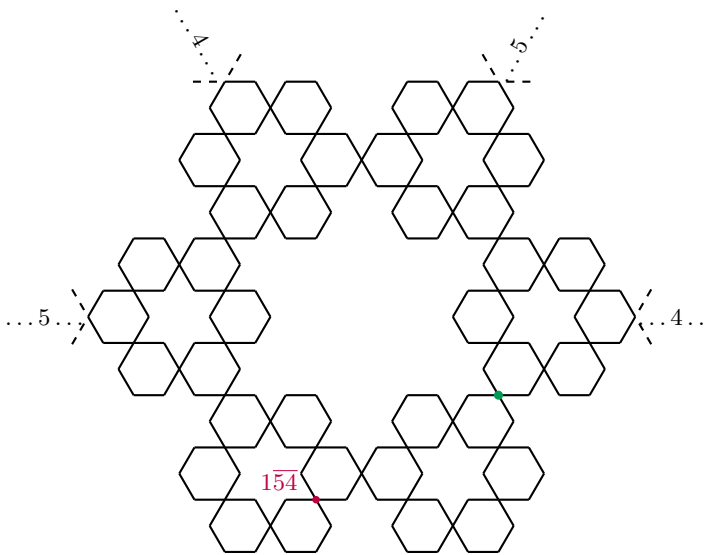

\section{Isomorphism classification} \label{sec:2}
Consider the $r$-cycle graph $P_r$ and its isometry group given by the dihedral group $D_r$ seen as a subgroup of $\mathrm{Sym}(V_r)$. Given an infinite sequence $\xi=x_1x_2\cdots$ in $V_r^{\infty}$, we define the action of $\sigma \in D_r$ on $V_r^{\infty}$ by
$$
\sigma(\xi)=\sigma(x_1)\sigma(x_2)\cdots.
$$

\begin{remark} \label{rmk:dihedral}
It is easy to see that this induces an isomorphism $ (\Gamma_k^{(r)},\xi[k]) \rightarrow (\Gamma_k^{(r)},\sigma(\xi[k]))$ of pointed graphs and, in fact, the group of all isomorphisms of $(\Gamma_k^{(r)},\xi[k])\rightarrow (\Gamma_k^{(r)},\sigma(\xi[k]))$ is isomorphic to $D_r$. To better visualize this, one can think of the permutations of the $r$ copies of $\Gamma_{k-1}$, which would provide exactly $D_r$ and then easily verify that, for compatibility, copies of $\Gamma_{k'}$ with $k' < k-1$ do not play a role in the definition of further isomorphisms.
\end{remark}

Two infinite sequences $\xi=x_1x_2\cdots$ and $\eta=y_1y_2\cdots$ are \textbf{cofinal} if there exists an $N\in \mathbb{N}$ such that $x_i=y_i$ for any $i\geq N$. Being cofinal defines an equivalence relation and we write $\xi \sim \eta$.\
With a slight abuse of notation let us regard an infinite sequence as a vertex of a limit graph $\Gamma_\xi$.\

As already mentioned above, the following result is shared with infinite Schreier graphs of bounded automata groups.

\begin{lemma} \label{lem:in}
The sequence $\eta$ represents a vertex in $\Gamma_\xi$ if and only if $\eta\sim \xi$. Therefore, in this case, we have $\Gamma_\xi = \Gamma_\eta$.
\end{lemma}

\begin{proof}
Suppose that $\eta\sim \xi$ and $\xi=x_{1} x_{2} \ldots$, that is for some $\widetilde{k}$ we have $\eta=y_1\ldots y_{\widetilde{k}} x_{\widetilde{k}+1}\ldots$. Informally, this is equivalent to say that, after a finite amount of steps, we are always picking the same copy where to be. More precisely, by construction of the graphs, $\eta[k]$ is contained in a ball of $\Gamma_{\xi[k]}$ centered in $\xi[k]$ with radius $\diam(\Gamma_{\widetilde{k}})$ for all $k\geq \widetilde{k}$. Hence $\eta$ belongs to $\Gamma_{\xi}$ by Gromov-Hausdorff convergence. 
\end{proof}

We have remarked that to define the limit graph we must consider it in the space of pointed graphs. Once we have the infinite (pointed) graph we might forget the distinguished vertex and look at the structure of the graph.

We are interested in the isomorphism problem of such infinite graphs. \\

\begin{lemma}\label{lem:iso}
The map $\sigma: (\Gamma_{\xi},\xi) \rightarrow (\Gamma_{\sigma(\xi)},\sigma(\xi)) $ is an isomorphism of pointed graphs. Moreover, if  $\phi$ is an isomorphisms between two pointed graphs, then $\phi=\sigma$ for some $\sigma \in D_{r}$.
\end{lemma}
\begin{proof}
The first part follows directly by definition, $\phi=\sigma$ is the limit of $\{\phi_k\}_{k=1}^{\infty}$ where $\phi_{k}:  (\Gamma_k, \xi[k]) \rightarrow  (\Gamma_k, \sigma(\xi)[k])$ is the pointed isometry defined by $\sigma$. By definition of Gromov-Hausdorff convergence, we conclude that $\phi$ is an isometry of the limits. \\

For the second part, let $\phi: (\Gamma_{\xi},\xi) \rightarrow (\Gamma_{\eta},\eta) $ be a pointed graph isomorphism. This means that $B(\xi,\ell)$ is isomorphic to $B(\eta,\ell)$ for all $\ell>0$. 
Fix an $\ell>0$ then, by the Gromov-Hausdorff convergence there exists $L$ such that for all $m \geq L$ we have $B_m(\xi[m],\ell) \simeq B(\xi,\ell)$ and $B_m(\eta[m],\ell) \simeq B(\eta,\ell)$. Composing all the pointed isomorphisms, we have an isomorphism $\widehat{\phi}_{\ell,m}:B_m(\xi[m],\ell) \rightarrow B_m(\eta[m],\ell)$. Note that these pointed isomorphisms are each compatible to one another, i.e. the restriction of $\phi_{\ell,m}$ to the ball $B_m(\xi[m],\ell-1)$ is equivalent to $\phi_{\ell-1,m}$ and the restriction of $\phi_{\ell,m}$ to the isomorphic copy of $B_{m-1}(\xi[m-1],\ell)$ embedded in $B_m(\xi[m],\ell)$ is equivalent to $\phi_{\ell,m-1}$. Now an isomorphic copy of $(\Gamma_k,\xi[k])$ is contained in $B_m(\xi[m],\ell)$ for some suitable $k$ and the family of pointed isomorphisms $\widehat{\phi}_{\ell,m}$ restricted to each $(\Gamma_k,\xi[k])$ provides a family of pointed isomorphisms $\widehat{\phi}_{k}:(\Gamma_k,\xi[k]) \rightarrow (\Gamma_k,\eta[k])$, so by \cref{rmk:dihedral} we have $\widehat{\phi}_{k}=\sigma$ for some $\sigma \in D_r$. 
Compatibility ensure that the isomorphism $\widehat{\phi}_{k}$ is induced by the same $\sigma$ for all $k$. Arguing as in the first part, we get that $\phi=\sigma$ as requested.
\end{proof}
\begin{theorem}\label{thm:iso}
Two graphs $\Gamma_\xi$ and $\Gamma_\eta$ are isomorphic if and only if there exists $\sigma\in D_r$ such that $\eta\sim \sigma(\xi)$.
\end{theorem}

\begin{proof}

By \cref{lem:iso}, the isomorphism between $\Gamma_\xi$ and $\Gamma_\eta $ is described by some $\sigma \in D_{r}$. Hence $\sigma(\xi) \in \Gamma_\eta $ and by \cref{lem:in} we conclude that $\sigma(\xi) \sim \eta$. \
On the other hand, suppose that $\sigma$ is such that $\sigma(\xi) \sim \eta$. Again by \cref{lem:iso}, the element $\sigma$ induces an isomorphism between $\Gamma_\xi $ and $\Gamma_{\sigma(\xi)} $. Finally, we have $\Gamma_{\sigma(\xi)} = \Gamma_\eta $ by \cref{lem:in}.
\end{proof}

\section{Horofunction boundary} \label{sec:3}
We need to recall some definitions and basic results. We fix a locally finite graph $X$ endowed with the standard path metric $\dist$ which identifies an edge with the unit interval.\

Let $C(X,\mathbb{Z})=\mathbb{Z}^X$ be the abelian group of all integer-valued function on $X$, which we endow with the product topology and where $\mathbb{Z}$ has the discrete topology.
Let us fix a base vertex $x_0 \in X$. For each vertex $v\in X$, we define
$$
\overline{\dist}_v(x) := \dist(x,v) - \dist(x_0,v), \qquad \forall x \in X.
$$
There is a well-defined embedding $\iota:X \hookrightarrow C(X,\mathbb{Z})$, defined as $\iota(v):=\overline{\dist}_v$. We denote by $\overline{\iota(X)}$ the closure of $\iota(X)$ in $C(X,\mathbb{Z})$.

\begin{definition}
The space $\overline{\iota(X)}$ is called the \textbf{horofunction compactification} of $X$. The \textbf{horofunction boundary} of $X$ is the set $\overline{\iota(X)}\setminus \iota(X)$ and it is denoted by $\partial_h X$.
\end{definition}

We will refer to elements of $\partial_h X$ simply as horofunctions. \\

We want a geometric description of the horofunctions of the graph. We need the following definition (see \cite{Rieffel02} for the general definition in metric spaces).
\begin{definition}\label{def:geodesic}
Let $X$ be a locally finite graph, and let $T$ be an unbounded subset of $\mathbb{N}$ containing $0$. Let $\gamma: T \to X$ such that $\gamma(t)$ is a vertex for each $t \in T$ and such that $\dist(\gamma(t),\gamma(0))=t$. Then
\begin{enumerate}
\item $\gamma$ is a \textbf{geodesic ray} if $\dist(\gamma(t), \gamma(s)) = |t - s|$, for all $t, s \in T$;
\item $\gamma$ is an \textbf{almost-geodesic ray} if there exists an integer $N$ such that, for every $t, s \in T$
with $t \geq s \geq N$, one has:
$$
|\dist(\gamma(t), \gamma(s)) + \dist(\gamma(s), \gamma(0)) - t| =0;
$$
\item $\gamma$ is a \textbf{weakly-geodesic ray} if, for every $y \in X$, 
there exists an integer $N$ such that, for every $s, t \geq N$, one has:
$$
|\dist(\gamma(t), y) - \dist(\gamma(s), y) - (t - s)| = 0.
$$
\end{enumerate}
\end{definition}

It is easy to see that the following implications hold
$$\text{geodesic ray } \Longrightarrow \text{ almost-geodesic ray } \Longrightarrow \text{ weakly-geodesic ray.}$$

\noindent The following result is a reformulation of \cite[Theorem 4.7]{Rieffel02} in the context of locally finite graphs.

\begin{theorem}
Let $X$ be a locally finite graph and let $\gamma$ be a weakly-geodesic ray in $X$. Then $\lim_{t\to +\infty} \overline{\dist}_{\gamma(t)}(y)$ pointwise converges to a horofunction.  Conversely, every horofunction is a pointwise limit of a weakly-geodesic ray.
\end{theorem}

The previous theorem helps us describing horofunctions from a more geometric point of view. We are now ready to give the last definition.
\begin{definition}
A horofunction defined by an almost-geodesic ray is called a \textbf{Busemann point}.
\end{definition}

In order to study horofunctions in our self-similar graphs, we will need to consider \textit{holes} and \textit{cut points}. We start with the following.

\begin{definition}
Let $\xi \in V_r^{\infty}$. We say that $\Gamma_\xi$ \textbf{grows away from} $j \in \{1, \ldots, r\}$ if $j$ appears infinitely many times in the infinite sequence $\xi$. 
\end{definition}

Note that this notion is well-defined for a graph. In fact, due to \cref{lem:in}, it is independent of the choice of $\xi$.

\begin{definition}

A $k$-\textbf{hole} is the subgraph of $\Gamma_k$ isomorphic to the $k$-iteration of a generalized Koch snowflake defining the central hole
of the graph. A bit more formally, such a subgraph is obtained by joining the gluing vertices with suitable curves made by replacing the edges of the $(k-1)$-hole with geodesics between their endpoints in the graph $\Gamma_{k-1}$.
\end{definition}
A formal and explicit definition of such a curve exists, but it is rather technical and not relevant for our purposes (see \cref{fig:hole} for graphical aid).
A copy of $\Gamma_{k-1}$ is said to \textbf{facing} a $k$-hole if its intersection with the hole is non trivial. A \textbf{hole} $H$ in $\Gamma_\xi$ is the $k$-hole of $\Gamma_{\xi[k]}$ for some $k \in \mathbb{N}$.\\

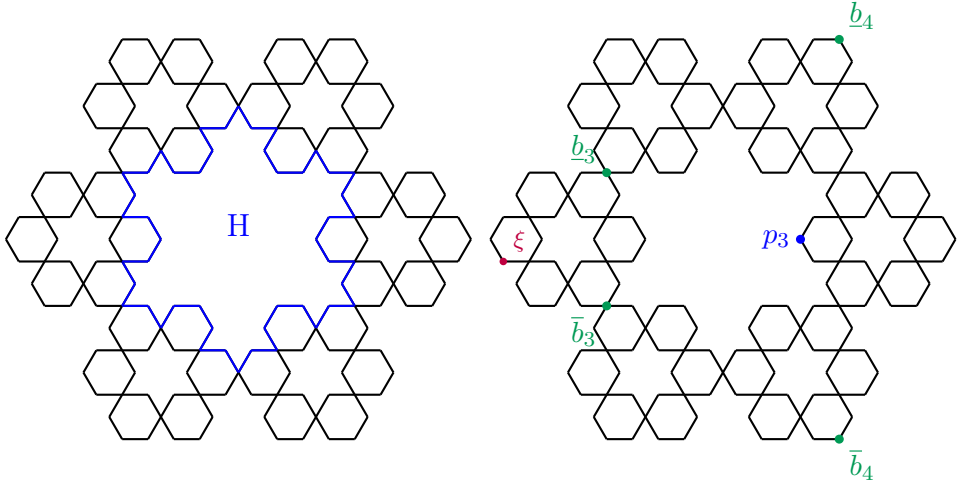
\begin{figure}
    \centering
\begin{tikzpicture}
\begin{scope}[scale=0.85]
\begin{scope}[scale=0.2]
\begin{scope}[xshift=8cm,yshift=15cm]
\draw[blue] (0,0) node[]{\large H};
\end{scope}
\begin{scope} 
  \foreach \i in {0,60,...,300} {
    \draw[thick] (\i:2) -- (\i+60:2);
  }
\end{scope}
\begin{scope}[xshift=4cm] 
  \foreach \i in {0,60,...,300} {
    \draw[thick] (\i:2) -- (\i+60:2);
  }
\end{scope}
\begin{scope}[xshift=-2cm,yshift=3.464cm] 
  \foreach \i in {0,60,...,300} {
    \draw[thick] (\i:2) -- (\i+60:2);
  }
\end{scope}
\begin{scope}[yshift=6.928cm] 
  \foreach \i in {0,60,...,300} {
    \draw[thick] (\i:2) -- (\i+60:2);
  }
  \draw[thick,blue] (120:2) -- (60:2);
  \draw[thick,blue] (0:2) -- (60:2);
\end{scope}
\begin{scope}[xshift=4cm,yshift=6.928cm] 
  \foreach \i in {0,60,...,300} {
    \draw[thick] (\i:2) -- (\i+60:2);
  }
  \draw[thick,blue] (120:2) -- (60:2);
  \draw[thick,blue] (0:2) -- (60:2);
  \draw[thick,blue] (0:2) -- (300:2);
  \draw[thick,blue] (120:2) -- (180:2);
\end{scope}
\begin{scope}[xshift=6cm,yshift=3.464cm] 
  \foreach \i in {0,60,...,300} {
    \draw[thick] (\i:2) -- (\i+60:2);
  }
  \draw[thick,blue] (120:2) -- (60:2);
  \draw[thick,blue] (0:2) -- (60:2);
\end{scope} 
\end{scope}

\begin{scope}[scale=0.2,xshift=12cm] 
\begin{scope} 
  \foreach \i in {0,60,...,300} {
    \draw[thick] (\i:2) -- (\i+60:2);
  }
\end{scope}
\begin{scope}[xshift=4cm] 
  \foreach \i in {0,60,...,300} {
    \draw[thick] (\i:2) -- (\i+60:2);
  }
\end{scope}
\begin{scope}[xshift=-2cm,yshift=3.464cm] 
  \foreach \i in {0,60,...,300} {
    \draw[thick] (\i:2) -- (\i+60:2);
  }
  \draw[thick,blue] (120:2) -- (180:2);
  \draw[thick,blue] (120:2) -- (60:2);
\end{scope}
\begin{scope}[yshift=6.928cm] 
  \foreach \i in {0,60,...,300} {
    \draw[thick] (\i:2) -- (\i+60:2);
  }
  \draw[thick,blue] (180:2) -- (240:2);
  \draw[thick,blue] (120:2) -- (180:2);
  \draw[thick,blue] (120:2) -- (60:2);
  \draw[thick,blue] (0:2) -- (60:2);
\end{scope}
\begin{scope}[xshift=4cm,yshift=6.928cm] 
  \foreach \i in {0,60,...,300} {
    \draw[thick] (\i:2) -- (\i+60:2);
  }
  \draw[thick,blue] (120:2) -- (180:2);
  \draw[thick,blue] (120:2) -- (60:2);
\end{scope}
\begin{scope}[xshift=6cm,yshift=3.464cm] 
  \foreach \i in {0,60,...,300} {
    \draw[thick] (\i:2) -- (\i+60:2);
  }
\end{scope} 
\end{scope}

\begin{scope}[scale=0.2,xshift=-6cm,yshift=10.392cm] 
\begin{scope} 
  \foreach \i in {0,60,...,300} {
    \draw[thick] (\i:2) -- (\i+60:2);
  }
\end{scope}
\begin{scope}[xshift=4cm] 
  \foreach \i in {0,60,...,300} {
    \draw[thick] (\i:2) -- (\i+60:2);
  }
  \draw[thick,blue] (0:2) -- (60:2);
  \draw[thick,blue] (0:2) -- (300:2);
\end{scope}
\begin{scope}[xshift=-2cm,yshift=3.464cm] 
  \foreach \i in {0,60,...,300} {
    \draw[thick] (\i:2) -- (\i+60:2);
  }
\end{scope}
\begin{scope}[yshift=6.928cm] 
  \foreach \i in {0,60,...,300} {
    \draw[thick] (\i:2) -- (\i+60:2);
  }
\end{scope}
\begin{scope}[xshift=4cm,yshift=6.928cm] 
  \foreach \i in {0,60,...,300} {
    \draw[thick] (\i:2) -- (\i+60:2);
  }
  \draw[thick,blue] (0:2) -- (60:2);
  \draw[thick,blue] (0:2) -- (300:2);
\end{scope}
\begin{scope}[xshift=6cm,yshift=3.464cm] 
  \foreach \i in {0,60,...,300} {
    \draw[thick] (\i:2) -- (\i+60:2);
  }
  \draw[thick,blue] (120:2) -- (60:2);
  \draw[thick,blue] (0:2) -- (60:2);
  \draw[thick,blue] (0:2) -- (300:2);
  \draw[thick,blue] (300:2) -- (240:2);
\end{scope} 
\end{scope}

\begin{scope}[scale=0.2,yshift=20.784cm] 
\begin{scope} 
  \foreach \i in {0,60,...,300} {
    \draw[thick] (\i:2) -- (\i+60:2);
  }
  \draw[thick,blue] (240:2) -- (300:2);
  \draw[thick,blue] (300:2) -- (0:2);
\end{scope}
\begin{scope}[xshift=4cm] 
  \foreach \i in {0,60,...,300} {
    \draw[thick] (\i:2) -- (\i+60:2);
  }
    \draw[thick,blue] (240:2) -- (300:2);
  \draw[thick,blue] (180:2) -- (240:2);
  \draw[thick,blue] (0:2) -- (60:2);
  \draw[thick,blue] (300:2) -- (0:2);
\end{scope}
\begin{scope}[xshift=-2cm,yshift=3.464cm] 
  \foreach \i in {0,60,...,300} {
    \draw[thick] (\i:2) -- (\i+60:2);
  }
\end{scope}
\begin{scope}[yshift=6.928cm] 
  \foreach \i in {0,60,...,300} {
    \draw[thick] (\i:2) -- (\i+60:2);
  }
\end{scope}
\begin{scope}[xshift=4cm,yshift=6.928cm] 
  \foreach \i in {0,60,...,300} {
    \draw[thick] (\i:2) -- (\i+60:2);
  }
\end{scope}
\begin{scope}[xshift=6cm,yshift=3.464cm] 
  \foreach \i in {0,60,...,300} {
    \draw[thick] (\i:2) -- (\i+60:2);
  }
  \draw[thick,blue] (240:2) -- (300:2);
  \draw[thick,blue] (300:2) -- (0:2);
\end{scope} 
\end{scope}

\begin{scope}[scale=0.2,xshift=12cm,yshift=20.784cm] 
\begin{scope} 
  \foreach \i in {0,60,...,300} {
    \draw[thick] (\i:2) -- (\i+60:2);
  }
  \draw[thick,blue] (120:2) -- (180:2);
  \draw[thick,blue] (180:2) -- (240:2);
  \draw[thick,blue] (240:2) -- (300:2);
  \draw[thick,blue] (300:2) -- (0:2);
\end{scope}
\begin{scope}[xshift=4cm] 
  \foreach \i in {0,60,...,300} {
    \draw[thick] (\i:2) -- (\i+60:2);
  }
  \draw[thick,blue] (240:2) -- (300:2);
  \draw[thick,blue] (180:2) -- (240:2);
\end{scope}
\begin{scope}[xshift=-2cm,yshift=3.464cm] 
  \foreach \i in {0,60,...,300} {
    \draw[thick] (\i:2) -- (\i+60:2);
  }
  \draw[thick,blue] (240:2) -- (300:2);
  \draw[thick,blue] (180:2) -- (240:2);
\end{scope}
\begin{scope}[yshift=6.928cm] 
  \foreach \i in {0,60,...,300} {
    \draw[thick] (\i:2) -- (\i+60:2);
  }
\end{scope}
\begin{scope}[xshift=4cm,yshift=6.928cm] 
  \foreach \i in {0,60,...,300} {
    \draw[thick] (\i:2) -- (\i+60:2);
  }
\end{scope}
\begin{scope}[xshift=6cm,yshift=3.464cm] 
  \foreach \i in {0,60,...,300} {
    \draw[thick] (\i:2) -- (\i+60:2);
  }
\end{scope} 
\end{scope}

\begin{scope}[scale=0.2,xshift=18cm,yshift=10.392cm] 
\begin{scope} 
  \foreach \i in {0,60,...,300} {
    \draw[thick] (\i:2) -- (\i+60:2);
  }
  \draw[thick,blue] (120:2) -- (180:2);
  \draw[thick,blue] (180:2) -- (240:2);
\end{scope}
\begin{scope}[xshift=4cm] 
  \foreach \i in {0,60,...,300} {
    \draw[thick] (\i:2) -- (\i+60:2);
  }
\end{scope}
\begin{scope}[xshift=-2cm,yshift=3.464cm] 
  \foreach \i in {0,60,...,300} {
    \draw[thick] (\i:2) -- (\i+60:2);
  }
  \draw[thick,blue] (60:2) -- (120:2);
  \draw[thick,blue] (180:2) -- (240:2);
  \draw[thick,blue] (120:2) -- (180:2);
  \draw[thick,blue] (240:2) -- (300:2);
\end{scope}
\begin{scope}[yshift=6.928cm] 
  \foreach \i in {0,60,...,300} {
    \draw[thick] (\i:2) -- (\i+60:2);
  }
  \draw[thick,blue] (120:2) -- (180:2);
  \draw[thick,blue] (180:2) -- (240:2);
\end{scope}
\begin{scope}[xshift=4cm,yshift=6.928cm] 
  \foreach \i in {0,60,...,300} {
    \draw[thick] (\i:2) -- (\i+60:2);
  }
\end{scope}
\begin{scope}[xshift=6cm,yshift=3.464cm] 
  \foreach \i in {0,60,...,300} {
    \draw[thick] (\i:2) -- (\i+60:2);
  }
\end{scope} 
\end{scope}

\end{scope}

\begin{scope}[scale=0.85,xshift=7.5cm]
\begin{scope}[scale=0.2]
\begin{scope} 
  \foreach \i in {0,60,...,300} {
    \draw[thick] (\i:2) -- (\i+60:2);
  }
\end{scope}
\begin{scope}[xshift=4cm] 
  \foreach \i in {0,60,...,300} {
    \draw[thick] (\i:2) -- (\i+60:2);
  }
\end{scope}
\begin{scope}[xshift=-2cm,yshift=3.464cm] 
  \foreach \i in {0,60,...,300} {
    \draw[thick] (\i:2) -- (\i+60:2);
  }
\end{scope}
\begin{scope}[yshift=6.928cm] 
  \foreach \i in {0,60,...,300} {
    \draw[thick] (\i:2) -- (\i+60:2);
  }
\end{scope}
\begin{scope}[xshift=4cm,yshift=6.928cm] 
  \foreach \i in {0,60,...,300} {
    \draw[thick] (\i:2) -- (\i+60:2);
  }
\end{scope}
\begin{scope}[xshift=6cm,yshift=3.464cm] 
  \foreach \i in {0,60,...,300} {
    \draw[thick] (\i:2) -- (\i+60:2);
  }
\end{scope} 
\end{scope}

\begin{scope}[scale=0.2,xshift=12cm] 
\begin{scope} 
  \foreach \i in {0,60,...,300} {
    \draw[thick] (\i:2) -- (\i+60:2);
  }
\end{scope}
\begin{scope}[xshift=4cm] 
  \foreach \i in {0,60,...,300} {
    \draw[thick] (\i:2) -- (\i+60:2);
  }
      \draw[ForestGreen] (300:2) node[circle,inner sep=1.2pt,fill]{} node[below right]{$\overline{b}_{4}$};
\end{scope}
\begin{scope}[xshift=-2cm,yshift=3.464cm] 
  \foreach \i in {0,60,...,300} {
    \draw[thick] (\i:2) -- (\i+60:2);
  }
\end{scope}
\begin{scope}[yshift=6.928cm] 
  \foreach \i in {0,60,...,300} {
    \draw[thick] (\i:2) -- (\i+60:2);
  }
\end{scope}
\begin{scope}[xshift=4cm,yshift=6.928cm] 
  \foreach \i in {0,60,...,300} {
    \draw[thick] (\i:2) -- (\i+60:2);
  }
\end{scope}
\begin{scope}[xshift=6cm,yshift=3.464cm] 
  \foreach \i in {0,60,...,300} {
    \draw[thick] (\i:2) -- (\i+60:2);
  }
\end{scope} 
\end{scope}

\begin{scope}[scale=0.2,xshift=-6cm,yshift=10.392cm] 
\begin{scope} 
  \foreach \i in {0,60,...,300} {
    \draw[thick] (\i:2) -- (\i+60:2);
  }
\end{scope}
\begin{scope}[xshift=4cm] 
  \foreach \i in {0,60,...,300} {
    \draw[thick] (\i:2) -- (\i+60:2);
  }
    \draw[ForestGreen] (300:2) node[circle,inner sep=1.2pt,fill]{} node[below left]{$\overline{b}_{3}$};
\end{scope}
\begin{scope}[xshift=-2cm,yshift=3.464cm] 
  \foreach \i in {0,60,...,300} {
    \draw[thick] (\i:2) -- (\i+60:2);
  }
   \draw[purple] (240:2) node[circle,inner sep=1pt,fill]{} node[above right]{\footnotesize $\xi$};
\end{scope}
\begin{scope}[yshift=6.928cm] 
  \foreach \i in {0,60,...,300} {
    \draw[thick] (\i:2) -- (\i+60:2);
  }
\end{scope}
\begin{scope}[xshift=4cm,yshift=6.928cm] 
  \foreach \i in {0,60,...,300} {
    \draw[thick] (\i:2) -- (\i+60:2);
  }
\end{scope}
\begin{scope}[xshift=6cm,yshift=3.464cm] 
  \foreach \i in {0,60,...,300} {
    \draw[thick] (\i:2) -- (\i+60:2);
  }
\end{scope} 
\end{scope}

\begin{scope}[scale=0.2,yshift=20.784cm] 
\begin{scope} 
  \foreach \i in {0,60,...,300} {
    \draw[thick] (\i:2) -- (\i+60:2);
  }
 \draw[ForestGreen] (240:2) node[circle,inner sep=1.2pt,fill]{} node[above left]{$\underline{b}_{3}$};
\end{scope}
\begin{scope}[xshift=4cm] 
  \foreach \i in {0,60,...,300} {
    \draw[thick] (\i:2) -- (\i+60:2);
  }
\end{scope}
\begin{scope}[xshift=-2cm,yshift=3.464cm] 
  \foreach \i in {0,60,...,300} {
    \draw[thick] (\i:2) -- (\i+60:2);
  }
\end{scope}
\begin{scope}[yshift=6.928cm] 
  \foreach \i in {0,60,...,300} {
    \draw[thick] (\i:2) -- (\i+60:2);
  }
\end{scope}
\begin{scope}[xshift=4cm,yshift=6.928cm] 
  \foreach \i in {0,60,...,300} {
    \draw[thick] (\i:2) -- (\i+60:2);
  }
\end{scope}
\begin{scope}[xshift=6cm,yshift=3.464cm] 
  \foreach \i in {0,60,...,300} {
    \draw[thick] (\i:2) -- (\i+60:2);
  }
\end{scope} 
\end{scope}

\begin{scope}[scale=0.2,xshift=12cm,yshift=20.784cm] 
\begin{scope} 
  \foreach \i in {0,60,...,300} {
    \draw[thick] (\i:2) -- (\i+60:2);
  }
\end{scope}
\begin{scope}[xshift=4cm] 
  \foreach \i in {0,60,...,300} {
    \draw[thick] (\i:2) -- (\i+60:2);
  }
\end{scope}
\begin{scope}[xshift=-2cm,yshift=3.464cm] 
  \foreach \i in {0,60,...,300} {
    \draw[thick] (\i:2) -- (\i+60:2);
  }
\end{scope}
\begin{scope}[yshift=6.928cm] 
  \foreach \i in {0,60,...,300} {
    \draw[thick] (\i:2) -- (\i+60:2);
  }
\end{scope}
\begin{scope}[xshift=4cm,yshift=6.928cm] 
  \foreach \i in {0,60,...,300} {
    \draw[thick] (\i:2) -- (\i+60:2);
  }
      \draw[ForestGreen] (60:2) node[circle,inner sep=1.2pt,fill]{} node[above right]{$\underline{b}_{4}$};
\end{scope}
\begin{scope}[xshift=6cm,yshift=3.464cm] 
  \foreach \i in {0,60,...,300} {
    \draw[thick] (\i:2) -- (\i+60:2);
  }
\end{scope} 
\end{scope}

\begin{scope}[scale=0.2,xshift=18cm,yshift=10.392cm] 
\begin{scope} 
  \foreach \i in {0,60,...,300} {
    \draw[thick] (\i:2) -- (\i+60:2);
  }
\end{scope}
\begin{scope}[xshift=4cm] 
  \foreach \i in {0,60,...,300} {
    \draw[thick] (\i:2) -- (\i+60:2);
  }
\end{scope}
\begin{scope}[xshift=-2cm,yshift=3.464cm] 
  \foreach \i in {0,60,...,300} {
    \draw[thick] (\i:2) -- (\i+60:2);
  }
 \draw[blue] (180:2) node[circle,inner sep=1.2pt,fill]{} node[left]{$p_{3}$};
\end{scope}
\begin{scope}[yshift=6.928cm] 
  \foreach \i in {0,60,...,300} {
    \draw[thick] (\i:2) -- (\i+60:2);
  }
\end{scope}
\begin{scope}[xshift=4cm,yshift=6.928cm] 
  \foreach \i in {0,60,...,300} {
    \draw[thick] (\i:2) -- (\i+60:2);
  }
\end{scope}
\begin{scope}[xshift=6cm,yshift=3.464cm] 
  \foreach \i in {0,60,...,300} {
    \draw[thick] (\i:2) -- (\i+60:2);
  }
\end{scope} 
\end{scope}

\end{scope}

\end{tikzpicture}
    \caption{On the left side, a $3$-hole with six facing copies of $\Gamma_3^{(6)}$. On the right side, the points $\overline{b}_m$ and $\underline{b}_m$ for $\xi=433$ and $i=3$, together with the antipodal point of the hole $H$ in the left figure with respect to $\xi$.}
    \label{fig:hole}
\end{figure}
Recall that $\Gamma_k$ contains $r$ copies of $\Gamma_{k-1}$ numbered from $0$ to $r-1$. If $n$ is even, we say that \textbf{antipodal copy} of the $i$-th copy is the $(i+r/2)$-th copy where $i+r/2$ is taken modulo $r$. If $r$ is odd, we say that there are two antipodal copies, the $(i+\lfloor r/2\rfloor)$-th and the $[i+\lceil r/2\rceil)$-th.
\begin{definition}
If $r$ is even, an \textbf{antipodal point} $p$ of a hole $H$ in $\Gamma_\xi$ with respect to $\xi$ is a vertex of $H$ such that $\xi$ and $p$ belong to the antipodal copies of $\Gamma_m$ facing the hole and such that $\dist(\overline{b}_m,p)=\dist(\underline{b}_m,p)$ where $\overline{b}_m$ and $\underline{b}_m$ are the two gluing vertices for the copy of $\Gamma_m$ containing $\xi$. More precisely if $\Gamma_m$ is in position $i$, then $\overline{b}_m$ glues $\Gamma_m$ to the $(i+1)$th-copy and $\underline{b}_m$ to the $(i-1)$th-copy (see again \cref{fig:hole} for graphical aid). 

If $r$ is odd, an antipodal point $p$ of a hole $H$ in $\Gamma_\xi$ with respect to $\xi$ is the gluing vertex of the two antipodal copies of the copy of $\Gamma_m$ containing $\xi$.
\end{definition}

Note that, in the latter case,
the condition $\dist(\overline{b}_m,p)=\dist(\underline{b}_m,p)$ still holds since they are the two gluing vertices (not equal to the antipodal point) of the two copies of $\Gamma_m$. 

\begin{remark}
\label{rem:antipodal}
If $\Gamma_\xi$ grows away from $j$, then there exists an unbounded subset $M$ of $\mathbb{N}$ such that $\xi$ belongs to the $j$-th copy of $\Gamma_{\xi[m]}$ for all $m\in M$. Call $H_{m}$ the copy of the $m$-hole in $\Gamma_{\xi[m]}$, so that $\diam(H_{m_1}) < \diam (H_{m_2})$ with $m_1 < m_2$ in $M$. Now, set $p_m$ to be the antipodal point of $H_m$ with respect to $\xi$, in particular $\xi$ and $p_m$ belong to antipodal copies of $\Gamma_{\xi[m]}$ and $\dist(\xi,p_m) \rightarrow \infty$ as $m \rightarrow \infty$.
\end{remark}
Exploiting the remark, we provide the following.
\begin{definition}
We call the sequence $\{p_m\}_{m \in M}$ obtained from \cref{rem:antipodal}, an \textbf{antipodal sequence}.   
\end{definition}

Before working with antipodal sequences, we first examine when certain sequences cannot be horofunctions, making use of holes.

\begin{lemma}
Let $\xi$ be such that $\Gamma_\xi$ grows away from $j$. Suppose that $\{q_m\}_{m \in M}$ is a sequence of vertices such that $\dist(\xi,q_m) \rightarrow \infty$ as $m \rightarrow \infty$ and suppose there exists two $\overline{b}_N$ and $\underline{b}_N$ related to a $j$ copy $\Gamma_{\xi[N]}$ such that there are two infinite subsequences $\{q_{\overline{m}}\}$ and $\{q_{\underline{m}}\}$ with $q_{\overline{m}}$ (resp. $q_{\underline{m}}$) such that a geodesic connecting $q_{\overline{m}}$ (resp. $q_{\underline{m}}$) to any vertex in $\Gamma_{\xi[N]}$ pass through $\overline{b}_N$ (resp. $\underline{b}_N$). Then $\{q_m\}_{m \in M}$ does not represent a horofunction.
\end{lemma}
\begin{proof}
Let us prove that $\{q_m\}_{m \in M}$ is not a weakly-geodesic rays.
We fix $y=\overline{b}_N$ and compute $$|\dist(q_{\overline{m}},\xi)-\dist(q_{\overline{m}} ,y)-\left( \dist(q_{\underline{m}},\xi)-\dist(q_{\underline{m}},y)\right)|.$$
By hypothesis, $$\dist(q_{\overline{m}},\xi)=\dist(q_{\overline{m}},\overline{b}_N)+\dist(\overline{b}_N,\xi) \text{ and } \dist(q_{\underline{m}},\overline{b}_N)=\dist(q_{\underline{m}},\underline{b}_N)+\dist(\underline{b}_N,\overline{b}_N).$$
Hence,
$$|\dist(q_{\overline{m}},\xi)-\dist(q_{\overline{m}},\overline{b}_N)-\left( \dist(q_{\underline{m}},\xi)-\dist(q_{\underline{m}},\overline{b}_N)\right)|=|\dist(\overline{b}_N,\xi)-\dist(\underline{b}_N,\xi)+\dist(\underline{b}_N,\overline{b}_N)|$$
and possibly choosing another vertex inside the smallest subgraph containing the three vertices $\xi, \underline{b}_N,\overline{b}_N$
(see \cref{lem:in}), one can get $\dist(\underline{b}_N,\xi)<\dist(\underline{b}_N,\overline{b}_N)+\dist(\overline{b}_N,\xi)$. 
\end{proof}

This means that we can distinguish three possible main types of weakly-geodesic rays: the ones that almost always pass through $\overline{b}_m$, the ones through $\underline{b}_m$ and the ones that almost always are at the same distance from $\overline{b}_m$ and $\underline{b}_m$. In particular, for $\xi=wj^\infty$, we want to show that in each of the first two types, which are symmetric due to the geometry of the limit graphs, there exists a unique Busemann point (see \cref{prop:two-busemann}), while the third type provides a non-Busemann point (see \cref{lem:antipodal-non-busemann}) together with countably many other non-Busemann points (see \cref{prop:infinite-non}).

\begin{proposition} \label{prop:two-busemann}
If $\Gamma_\xi$ is such that $\xi=wj^\infty$, then the boundary has exactly two Busemann points.
\end{proposition}
\begin{proof}
By \cref{lem:in}, we can assume $\xi=j^\infty$. Consider the antipodal sequence $\{p_m\}_{m \in M}$ and its associated gluing vertices sequences $\{\overline{b}_m\}_{m \in M}$ and $\{\underline{b}_m\}_{m \in M}$. We claim that these two sequences lie on two different geodesic rays. With a slight abuse of notation, we set $\dist(\overline{b}_m,\xi)=\dist(\underline{b}_m,\xi)=m$. We consider the path $\overline{\gamma}$ such that $\overline{\gamma}(m)=\overline{b}_m$ and $\overline{\gamma}([m_1,m_2])$ is a geodesic going from $\overline{b}_{m_1}$ to $\overline{b}_{m_2}$ where $m_1$ and $m_2$ are two consecutive indices in $M$. 
In a similar way, we construct the path $\underline{\gamma}$ passing through the points
of the sequence $\{\underline{b}_m\}_{m \in M}$. We have to show that the concatenation $\overline{\gamma}_{|[m_1,m_2]} \ast \overline{\gamma}_{|[m_2,m_3]}$ is a geodesic. By construction, $\overline{\gamma}([m_2,m_3])$ lies entirely in $\Gamma_{\xi[m_3]}$ and $\overline{\gamma}([m_2,m_3]) \cap \Gamma_{\xi[m_2]}=\{\overline{\gamma}(m_2)\}$. On the other hand, removing $\overline{\gamma}(m_2)$ and $\underline{\gamma}(m_2)$ from $\Gamma_{\xi[m_3]}$ disconnects $\Gamma_{\xi[m_2]}$ and this implies that a geodesic from $\overline{\gamma}(m_1)$ and $\overline{\gamma}(m_3)$ must pass through one of these points and its length must be the sum of the length of two geodesics, one from $\overline{\gamma}(m_1)$ to $p$ and the other from $p$ to $\overline{\gamma}(m_3)$, where $p \in \{\overline{b}_{m_2},\underline{b}_{m_2}\}$. We can exclude $p = \underline{b}_{m_2}$ because, by construction, 
$$
\dist(\overline{b}_{m_t}, \overline{b}_{m_2}) < \dist(\overline{b}_{m_t}, \underline{b}_{m_2}) \quad \text{for } t = 1, 3,
$$
and we now explain why. Suppose that $\overline{b}_{m_1}$ is gluing two copies of $\Gamma_k$ for some $k$. Then:
\begin{enumerate}
    \item To reach $\overline{b}_{m_2}$, one must pass through $f(r)$ copies of $\Gamma_k$;
    \item To reach $\underline{b}_{m_2}$, one must either:
    \begin{itemize}
        \item pass through $\underline{b}_{m_1}$ and then through $f(r)$ copies of $\Gamma_k$; or
        \item pass through more than $f(r)+1$ copies of $\Gamma_k$.
    \end{itemize}
\end{enumerate}
In both cases under point (2), the path requires at least $f(r)+1$ copies of $\Gamma_k$, hence
$$
\dist(\overline{b}_{m_t}, \overline{b}_{m_2}) < \dist(\overline{b}_{m_t}, \underline{b}_{m_2}),
$$
and the required geodesic does not pass through $\underline{b}_{m_2}$. Note that all the previous considerations on coarse distances are true because we are implicitly exploiting the simmetries of the case $\xi=j^\infty$. To conclude the proof, let us note that, since $\{\overline{b}_m\}_{m \in M}$ and $\{\underline{b}_m\}_{m \in M}$ are sequences of cut points, another infinite geodesic starting from $\xi$ should pass through infinitely many $\{\overline{b}_m\}_{m \in M}$ or infinitely many $\{\underline{b}_m\}_{m \in M}$ and hence should describe the same horofunction of either one of the two Busemann points already constructed.
\end{proof}

\begin{lemma} \label{lem:antipodal-non-busemann}
Let $\Gamma_\xi$ be such that $\xi=wj^\infty$. Then the antipodal sequence describes a horofunction, which is not a Busemann point.
\end{lemma}
\begin{proof}
First, we need to prove that $P:T \rightarrow \Gamma_\xi$ where $P(\dist(\xi,p_m)):=p_m$ is a weakly-geodesic ray, that is, if we fix a vertex $y \in \Gamma_\xi$, eventually
$$
|\dist(p_m, y) - \dist(p_n, y) - (\dist(\xi,p_m) - \dist(\xi,p_n))| =0.
$$

Let $N$ be the smallest $n$ such that $\xi$ and $y$ belong to $\Gamma_{\xi[n]}$, and $\Gamma_{\xi[n]}$ is the $j$ copy in $\Gamma_{\xi[n+1]}$. From now on, we consider $n,m \in T$ such that $n,m \geq N$. Let $b_m$ (resp. $b_n$) be (possibly the same) gluing vertex of $\Gamma_{\xi[N]}$ to the rest of $\Gamma_{\xi[N+1]}$ such that a geodesic connecting $\xi$ and $p_m$ (resp. $p_n$) passes through $b_m$ (resp. $b_n$), so that
$$
\dist(\xi, p_m) = \dist(\xi,b_m) +\dist(b_m, p_m) \text{ and } \dist(\xi, p_n) = \dist(\xi,b_n) +\dist(b_n, p_n).
$$
Analogously, let $\widetilde{b}_m$ and $\widetilde{b}_n$ be the gluing vertices such that a geodesic connecting $y$ and $p_m$ (resp. $p_n$) passes through $\widetilde{b}_m$ (resp. $\widetilde{b}_n$), again we have
$$
\dist(y, p_m) = \dist(y,\widetilde{b}_m) +\dist(\widetilde{b}_m, p_m) \text{ and } \dist(y, p_n) = \dist(y,\widetilde{b}_n) +\dist(\widetilde{b}_n, p_n).
$$
The definition of antipodal point implies that $\dist(p_N,b_N)=\dist(p_N,\widetilde{b}_N)$ with $\widetilde{b}_N$ as in the previous sentence lying on a geodesic between $y$ and $p_N$. We claim that $\dist(p_m,b_m)=\dist(p_m,\widetilde{b}_m)$ for all $m>N$. If $b_m=\widetilde{b}_m$ we are done, otherwise without loss of generality suppose that $b_m$ is closer to $\overline{b}_m$ and $\widetilde{b}_m$ is closer to $\underline{b}_m$. This means that 
$$\dist(p_m,b_m)=\dist(p_m,\overline{b}_m)+\dist(\overline{b}_m,b_m)=\dist(p_m,\underline{b}_m)+\dist(\underline{b}_m,\widetilde{b}_m)=\dist(p_m,\widetilde{b}_m),$$
where $\dist(p_m,\overline{b}_m)=\dist(p_m,\underline{b}_m)$ follows again by definition of antipodal point, while $\dist(\overline{b}_m,b_m)=\dist(\underline{b}_m,\widetilde{b}_m)$ follows from the assumption and the geometry of $\Gamma_k$ (see \cref{fig:parallel-geodesics}, these are the finite geodesics described in \cref{prop:two-busemann}). Again, here we are exploiting the symmetries of the case $\xi=wj^\infty$.  \\

\begin{figure}
    \centering
\begin{tikzpicture}
\begin{scope}[scale=0.2]
\begin{scope}[xshift=8cm,yshift=15cm]
\end{scope}
\begin{scope} 
  \foreach \i in {0,60,...,300} {
    \draw[thick] (\i:2) -- (\i+60:2);
  }
    \draw[ForestGreen] (0:2) node[circle,inner sep=1pt,fill]{} node[right]{\footnotesize $\widetilde{b}_4$};
    \draw[ForestGreen] (120:2) node[circle,inner sep=1pt,fill]{} node[above left]{\footnotesize $b_4$};
    \draw[purple] (300:2) node[circle,inner sep=1pt,fill]{} node[below]{\footnotesize $y$};
    \draw[purple] (180:2) node[circle,inner sep=1pt,fill]{} node[left]{\footnotesize $\xi$};
\end{scope}
\begin{scope}[xshift=4cm] 
  \foreach \i in {0,60,...,300} {
    \draw[thick] (\i:2) -- (\i+60:2);
  }
  \draw[thick,ForestGreen] (60:2) -- (120:2);
  \draw[thick,ForestGreen] (120:2) -- (180:2);
\end{scope}
\begin{scope}[xshift=-2cm,yshift=3.464cm] 
  \foreach \i in {0,60,...,300} {
    \draw[thick] (\i:2) -- (\i+60:2);
  }
  \draw[thick,ForestGreen] (300:2) -- (0:2);
  \draw[thick,ForestGreen] (0:2) -- (60:2);
\end{scope}
\begin{scope}[yshift=6.928cm] 
  \foreach \i in {0,60,...,300} {
    \draw[thick] (\i:2) -- (\i+60:2);
  }
  \draw[thick,ForestGreen] (240:2) -- (180:2);
  \draw[thick,ForestGreen] (120:2) -- (180:2);
\end{scope}
\begin{scope}[xshift=4cm,yshift=6.928cm] 
  \foreach \i in {0,60,...,300} {
    \draw[thick] (\i:2) -- (\i+60:2);
  }

\end{scope}
\begin{scope}[xshift=6cm,yshift=3.464cm] 
  \foreach \i in {0,60,...,300} {
    \draw[thick] (\i:2) -- (\i+60:2);
  }
  \draw[thick,ForestGreen] (240:2) -- (300:2);
  \draw[thick,ForestGreen] (300:2) -- (0:2);
\end{scope} 
\end{scope}

\begin{scope}[scale=0.2,xshift=12cm] 
\begin{scope} 
  \foreach \i in {0,60,...,300} {
    \draw[thick] (\i:2) -- (\i+60:2);
  }
\end{scope}
\begin{scope}[xshift=4cm] 
  \foreach \i in {0,60,...,300} {
    \draw[thick] (\i:2) -- (\i+60:2);
  }
\end{scope}
\begin{scope}[xshift=-2cm,yshift=3.464cm] 
  \foreach \i in {0,60,...,300} {
    \draw[thick] (\i:2) -- (\i+60:2);
  }
  \draw[thick,ForestGreen] (120:2) -- (180:2);
  \draw[thick,ForestGreen] (120:2) -- (60:2);
\end{scope}
\begin{scope}[yshift=6.928cm] 
  \foreach \i in {0,60,...,300} {
    \draw[thick] (\i:2) -- (\i+60:2);
  }
  \draw[thick,ForestGreen] (300:2) -- (240:2);
  \draw[thick,ForestGreen] (0:2) -- (300:2);
\end{scope}
\begin{scope}[xshift=4cm,yshift=6.928cm] 
  \foreach \i in {0,60,...,300} {
    \draw[thick] (\i:2) -- (\i+60:2);
  }
  \draw[thick,ForestGreen] (120:2) -- (180:2);
  \draw[thick,ForestGreen] (120:2) -- (60:2);
\end{scope}
\begin{scope}[xshift=6cm,yshift=3.464cm] 
  \foreach \i in {0,60,...,300} {
    \draw[thick] (\i:2) -- (\i+60:2);
  }
\end{scope} 
\end{scope}

\begin{scope}[scale=0.2,xshift=-6cm,yshift=10.392cm] 
\begin{scope} 
  \foreach \i in {0,60,...,300} {
    \draw[thick] (\i:2) -- (\i+60:2);
  }
\end{scope}
\begin{scope}[xshift=4cm] 
  \foreach \i in {0,60,...,300} {
    \draw[thick] (\i:2) -- (\i+60:2);
  }
  \draw[thick,ForestGreen] (0:2) -- (60:2);
  \draw[thick,ForestGreen] (0:2) -- (300:2);
\end{scope}
\begin{scope}[xshift=-2cm,yshift=3.464cm] 
  \foreach \i in {0,60,...,300} {
    \draw[thick] (\i:2) -- (\i+60:2);
  }
\end{scope}
\begin{scope}[yshift=6.928cm] 
  \foreach \i in {0,60,...,300} {
    \draw[thick] (\i:2) -- (\i+60:2);
  }
\end{scope}
\begin{scope}[xshift=4cm,yshift=6.928cm] 
  \foreach \i in {0,60,...,300} {
    \draw[thick] (\i:2) -- (\i+60:2);
  }
  \draw[thick,ForestGreen] (0:2) -- (60:2);
  \draw[thick,ForestGreen] (0:2) -- (300:2);
\end{scope}
\begin{scope}[xshift=6cm,yshift=3.464cm] 
  \foreach \i in {0,60,...,300} {
    \draw[thick] (\i:2) -- (\i+60:2);
  }
  \draw[thick,ForestGreen] (120:2) -- (180:2);
  \draw[thick,ForestGreen] (180:2) -- (240:2);
\end{scope} 
\end{scope}

\begin{scope}[scale=0.2,yshift=20.784cm] 
\begin{scope} 
  \foreach \i in {0,60,...,300} {
    \draw[thick] (\i:2) -- (\i+60:2);
  }
  \draw[thick,ForestGreen] (240:2) -- (180:2);
  \draw[thick,ForestGreen] (120:2) -- (180:2);
\end{scope}
\begin{scope}[xshift=4cm] 
  \foreach \i in {0,60,...,300} {
    \draw[thick] (\i:2) -- (\i+60:2);
  }
\end{scope}
\begin{scope}[xshift=-2cm,yshift=3.464cm] 
  \foreach \i in {0,60,...,300} {
    \draw[thick] (\i:2) -- (\i+60:2);
  }
  \draw[thick,ForestGreen] (0:2) -- (60:2);
  \draw[thick,ForestGreen] (0:2) -- (300:2);
\end{scope}
\begin{scope}[yshift=6.928cm] 
  \foreach \i in {0,60,...,300} {
    \draw[thick] (\i:2) -- (\i+60:2);
  }
  \draw[thick,ForestGreen] (180:2) -- (240:2);
  \draw[thick,ForestGreen] (180:2) -- (120:2);
  \draw[ForestGreen] (120:2) node[circle,inner sep=1pt,fill]{} node[above]{\footnotesize $\overline{b}_4$};
\end{scope}
\begin{scope}[xshift=4cm,yshift=6.928cm] 
  \foreach \i in {0,60,...,300} {
    \draw[thick] (\i:2) -- (\i+60:2);
  }
\end{scope}
\begin{scope}[xshift=6cm,yshift=3.464cm] 
  \foreach \i in {0,60,...,300} {
    \draw[thick] (\i:2) -- (\i+60:2);
  }
\end{scope} 
\end{scope}

\begin{scope}[scale=0.2,xshift=12cm,yshift=20.784cm] 
\begin{scope} 
  \foreach \i in {0,60,...,300} {
    \draw[thick] (\i:2) -- (\i+60:2);
  }

\end{scope}
\begin{scope}[xshift=4cm] 
  \foreach \i in {0,60,...,300} {
    \draw[thick] (\i:2) -- (\i+60:2);
  }
\end{scope}
\begin{scope}[xshift=-2cm,yshift=3.464cm] 
  \foreach \i in {0,60,...,300} {
    \draw[thick] (\i:2) -- (\i+60:2);
  }
\end{scope}
\begin{scope}[yshift=6.928cm] 
  \foreach \i in {0,60,...,300} {
    \draw[thick] (\i:2) -- (\i+60:2);
  }
\end{scope}
\begin{scope}[xshift=4cm,yshift=6.928cm] 
  \foreach \i in {0,60,...,300} {
    \draw[thick] (\i:2) -- (\i+60:2);
  }
\end{scope}
\begin{scope}[xshift=6cm,yshift=3.464cm] 
  \foreach \i in {0,60,...,300} {
    \draw[thick] (\i:2) -- (\i+60:2);
  }
\end{scope} 
\end{scope}

\begin{scope}[scale=0.2,xshift=18cm,yshift=10.392cm] 
\begin{scope} 
  \foreach \i in {0,60,...,300} {
    \draw[thick] (\i:2) -- (\i+60:2);
  }
  \draw[thick,ForestGreen] (240:2) -- (300:2);
  \draw[thick,ForestGreen] (300:2) -- (0:2);
\end{scope}
\begin{scope}[xshift=4cm] 
  \foreach \i in {0,60,...,300} {
    \draw[thick] (\i:2) -- (\i+60:2);
  }
   \draw[thick,ForestGreen] (180:2) -- (120:2);
  \draw[thick,ForestGreen] (120:2) -- (60:2);
\end{scope}
\begin{scope}[xshift=-2cm,yshift=3.464cm] 
  \foreach \i in {0,60,...,300} {
    \draw[thick] (\i:2) -- (\i+60:2);
  }
\end{scope}
\begin{scope}[yshift=6.928cm] 
  \foreach \i in {0,60,...,300} {
    \draw[thick] (\i:2) -- (\i+60:2);
  }
\end{scope}
\begin{scope}[xshift=4cm,yshift=6.928cm] 
  \foreach \i in {0,60,...,300} {
    \draw[thick] (\i:2) -- (\i+60:2);
  }
\end{scope}
\begin{scope}[xshift=6cm,yshift=3.464cm] 
  \foreach \i in {0,60,...,300} {
    \draw[thick] (\i:2) -- (\i+60:2);
  }
  \draw[thick,ForestGreen] (240:2) -- (300:2);
  \draw[thick,ForestGreen] (300:2) -- (0:2);
  \draw[ForestGreen] (0:2) node[circle,inner sep=1pt,fill]{} node[right]{\footnotesize $\underline{b}_4$};
\end{scope} 
\end{scope}
\end{tikzpicture}
    \caption{Two geodesics connecting the gluing vertices as in the proof of \cref{lem:antipodal-non-busemann}: they have the same length.}
    \label{fig:parallel-geodesics}
\end{figure}
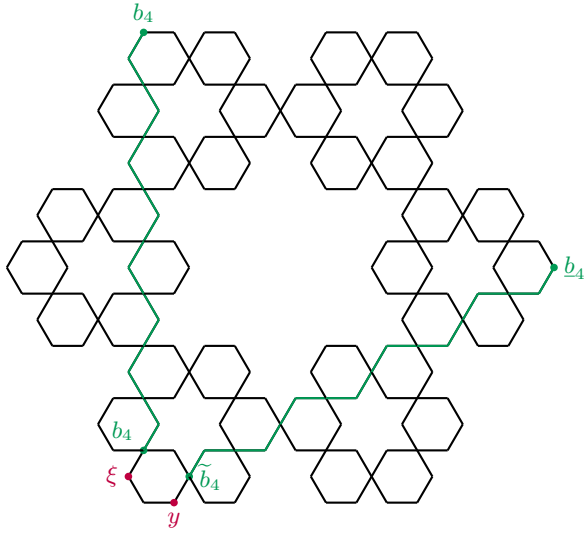
Combining this with all the previous equalities, we have
\begin{multline*}
|\dist(p_m, y) - \dist(\xi,p_m) - (\dist(p_n, y) - \dist(\xi,p_n))| =\\
|\dist(\widetilde{b}_m, y) - \dist(\xi,b_m) - (\dist(\widetilde{b}_n, y) - \dist(\xi,b_n))|.
\end{multline*}

By reasoning as before and recalling the symmetries of the case $\xi=wj^\infty$, we have $\dist(p_m,\underline{b}_N)=\dist(p_m,\overline{b}_N)$. This means that $b_m$ is the closest vertex to $\xi$ among $\underline{b}_N$ and $\overline{b}_N$, since this holds for any $m$, we have $b_m=b_n$. In the same way, $\widetilde{b}_m=\widetilde{b}_n$.
Hence

\begin{multline*}
|\dist(\widetilde{b}_m, y) - \dist(\xi,b_m) - (\dist(\widetilde{b}_n, y) - \dist(\xi,b_n))|=\\
|\dist(\widetilde{b}_m, y) - \dist(\xi,b_m) - (\dist(\widetilde{b}_m, y) - \dist(\xi,b_m))| =0
\end{multline*}

To conclude, we should prove that $P:T \rightarrow \Gamma_\xi$ is not a Busemann point. But this follows directly, almost verbatim, from the proof of the \cite[Proposition 4.2]{DD16}, taking $v_1$ and $v_2$ as $\underline{b}_N$ and $\overline{b}_N$ and noticing that $\varphi_y(x)=\overline{\dist}_x(y)$.
\end{proof}

Before continuing, it is worth mentioning that the horofunction described in the previous lemma can be seen as sequence of atoms (see \cite{BBM,Perego} for an introduction to the formalism) which are totally disconnected. It is straightforward to verify this, for example, when $\xi=\overline{4}$, $r=6$ and one computes the first level atoms (i.e. the atoms related to $B(\xi,1)$). The existence of disconnected atoms is yet to be understood in the hyperbolic and $\mathrm{CAT}(0)$ cases.

\begin{proposition}\label{prop:infinite-non}
If $\Gamma_\xi$ is such that $\xi=wj^\infty$, then $\partial_h \Gamma_\xi$ contains countably many non-Busemann points.
\end{proposition}
\begin{proof}
Let $y$ be any vertex in the graph $\Gamma_\xi$ and let $N$ be as in \cref{lem:antipodal-non-busemann}. Consider the antipodal sequence $\{p_m\}_{m \in M}$. If $A_m$ and $B_m$ are the vertices gluing the antipodal copy (or copies) of $p_m$ to the other copies in $\Gamma_m$, then by the simmetry of the space we have that $\dist(A_m,B_m)$ is even. Hence, one can consider $\widehat{p}_m$ the midpoint of a geodesic $\gamma_m$ realizing such distance and applying the same technique as in \cref{lem:antipodal-non-busemann}, one can show that $\{\widehat{p}_m\}_{m \in M}$ describes the same horofunction of the antipodal sequence.\

To retrieve a countable number of horofunctions, let us ``shift'' the sequence we have just constructed. We simply mean that if $\gamma_m (\widehat{t})=\widehat{p}_m$, then we define $p_{m,t}:=\gamma_m(\widehat{t}+t)$ with $t \in \mathbb{Z}$ such that $0<\widehat{t}+t<\dist(A_m,B_m)$. Moreover, we denote by $b_{m,t}$ (resp. $\widetilde{b}_{m,t}$) the vertex in $\{\overline{b}_N,\underline{b}_N\}$ lying on a geodesic connecting $p_{m,t}$ to $\xi$ (resp. $y$). Note that $b_{m,t}=b_{n,t}$ and similarly one sees that $\widetilde{b}_{m,t} = \widetilde{b}_{n,t}$, as we are about to show.
Indeed,
suppose that $\dist(A_m,p_{m,t})<\dist(B_m,p_{m,t})$, the other case being analogous, and note that then the same holds for $n$, that is, we must also have $\dist(A_n,p_{n,t})<\dist(B_n,p_{n,t})$. If $\dist(\xi,\overline{b}_N) \leq \dist(\xi,\underline{b}_N)$, then there is nothing to show. Conversely, the choice of one of the two vertices depend only on $t$. In fact, it is easy to see that $\dist(\underline{b}_N,p_{m,t})=\dist(\overline{b}_N,p_{m,t})+2|t|$ for all $m>N$.

In the same fashion as in \cref{lem:antipodal-non-busemann}, we get the following simplification
\begin{multline*}
|\dist(p_{m,t} , y) - \dist(p_{n,t} , y) - (\dist(\xi,p_{m,t}) - \dist(\xi,p_{n,t} ))| = \\
|\dist(b_{m,t},p_{m,t})-\dist(\widetilde{b}_{m,t},p_{m,t})-\left( \dist(b_{n,t},p_{n,t})-\dist(\widetilde{b}_{n,t},p_{n,t}) \right)|.
\end{multline*}

If $b_{m,t}=\widetilde{b}_{m,t}$ there is nothing to show. Otherwise, suppose without loss of generality that $\dist(b_{m,t},A_m) < \dist(b_{m,t},B_m)$, so that we must also have $\dist(\widetilde{b}_{m,t},A_m) > \dist(\widetilde{b}_{m,t},B_m)$.
Now, we can simplify once more by noticing that $\dist(b_{m,t},A_m)=\dist(\widetilde{b}_{m,t},B_m)$. Namely, we have
\begin{multline*}
|\dist(b_{m,t},p_{m,t})-\dist(\widetilde{b}_{m,t},p_{m,t})-\left( \dist(b_{n,t},p_{n,t})-\dist(\widetilde{b}_{n,t},p_{n,t}) \right)|\\
|\dist(A_m,p_{m,t})-\dist(B_m,p_{m,t})-\left( \dist(A_n,p_{n,t})-\dist(B_n,p_{n,t}) \right)|
\end{multline*}
Again, suppose without loss of generality that $\gamma_m(0)=A_m$ and $t>0$, then  $\dist(A_m,p_{m,t})=\dist(B_m,p_{m,t})+2t$ and the same holds for $n$. In the end the previous equality becomes
$$|\dist(A_m,p_{m,t})-\dist(B_m,p_{m,t})-\left( \dist(A_n,p_{n,t})-\dist(B_n,p_{n,t}) \right)|=|2t-2t|=0.$$

As a second step, we need to show that the sequences describe different horofunctions.\\
If $t_1t_2<0$, then we can assume that $p_{m,t_1}$ (resp. $p_{m,t_2}$) is nearer to $A_m$ (resp. $B_m$). Fix $y=A_m$ and determine the possible values of $t_1$ and $t_2$ such that 
$$\dist(A_m,p_{m,t_1})-\dist(\xi,p_{m,t_1})=\dist(A_m,p_{m,t_2})-\dist(\xi,p_{m,t_2}).$$

We thus suppose that $t_1,t_2$ satisfy the equation above. We observe that
$\dist(A_m,p_{m,t_1})-\dist(\xi,p_{m,t_1})=-\dist(\xi,A_m)$ and $\dist(\xi,p_{m,t_2})=\dist(\xi,A_m)+\dist(B_m,p_{m,t_2})$, which we substitute in the equation above to get that
$$\dist(A_m,p_{m,t_2})=\dist(B_m,p_{m,t_2}),$$ 
and so, by definition, this means that $p_{m,t_2}=\widehat{p}_m$. Fixing $y=B_m$ gives $p_{m,t_1}=\widehat{p}_m$. \

Now suppose again without loss of generality that $\gamma_m(0)=b_{m,t}$ and that $t_1\leq t_2 <0$. Fix $y$ to be the other gluing vertex different from $b_{m,t}$. We want to determine when 
$$\dist(b_{m,t},p_{m,t_1})-\dist(\xi,p_{m,t_1})=\dist(b_{m,t},p_{m,t_2})-\dist(\xi,p_{m,t_2}).$$

By construction $\dist(\xi,p_{m,t_2})=\dist(\xi,p_{m,t_1})+|t_1|-|t_2|$, while 

$$\dist(y,p_{m,t_1})=\dist(y,\widehat{p}_{m})+|t_1| \text{ and } \dist(y,p_{m,t_2})=\dist(y,\widehat{p}_{m})+|t_2|.$$

Combining all together, we end up with

$$|t_1|=|t_2|-(|t_1|-|t_2|)$$

which is true if and only if $t_1=t_2$.\\

All is left to do is to prove that they are not Busemann points, but, as in \cref{lem:antipodal-non-busemann}, this follows from \cite[Proposition 4.2]{DD16}.
\end{proof}

\begin{remark}
More precisely, the antipodal point $p_m$ and the point $\widehat{p}_m$ considered in the proof are in the same atom (as in the notation mentioned immediately before \cref{prop:infinite-non}).
\end{remark}

Combining together the two previous propositions, we have the following.

\begin{theorem} \label{mainthm}
The horofunction boundary of $\Gamma_\xi$ with $\xi=wj^\infty$ has exactly two Busemann points and contains countably many non-Busemann points. 
\end{theorem}

We strongly believe that our result, with the addition of further technical details, could be generalized to a generic $\Gamma_\xi$, henceforth we provide the following.
\begin{conjecture}
Let $\Gamma_\xi$ be such that it grows away from $j$ and define $S:=\{s \in V_r\ \mid s \neq j \text{ and } sj \text{ occurs infinitely many times in } \xi\}$.
\begin{enumerate}
    \item If $S$ is empty, we are in the conditions of Theorem \ref{mainthm}.
    \item If $s \in S$, then there are 2 Busemann points and
    \begin{enumerate}
        \item If $j+\dfrac{r-f(r)}{2}\leq s \leq j+\dfrac{r+f(r)}{2}$,  there exist countably many non-Busemann points.
        \item If $s < j+\dfrac{r-f(r)}{2}$, there exists countably many non-Busemann points.
        \item If $s> j+\dfrac{r+f(r)}{2}$, there exists countably many non-Busemann points.
    \end{enumerate}
    The cases (a), (b) and (c) give rise to different non-Busemann points.
\end{enumerate}
\end{conjecture}

\section*{Acknowledgements}
 The second author is a member of the PRIN 2022 ``Group theory and its applications'' research group and gratefully acknowledges the support of the PRIN project 2022-NAZ-0286. Moreover, the second author gratefully acknowledges the support of the Universit\`a degli Studi di Milano--Bicocca (FA project 2021-ATE-0033 ``Strutture Algebriche''). The third author acknowledges support from the Swiss Government Excellence Scholarship and from Swiss NSF grant 200020-200400. All the authors are members of the Gruppo Nazionale per le Strutture Algebriche, Geometriche e le loro Applicazioni (GNSAGA) of the Istituto Nazionale di Alta Matematica (INdAM).

\printbibliography[heading=bibintoc]

\end{document}